\documentclass[11pt]{amsart}

\usepackage{graphicx}
\usepackage{hyperref}
\usepackage{amsmath}
\usepackage{amsfonts}
\usepackage{amssymb}
\usepackage{array}

\usepackage{enumitem}

\SetLabelAlign{parright}{\parbox[t]{\labelwidth}{\raggedleft#1}}

\setlist[description]{align=parright,style=multiline}
\usepackage{environ}
\newlength{\cdescwidth}
\SetLabelAlign{CenterWithParen}{\makebox[\cdescwidth]{#1}}
\NewEnviron{cdesc}{%
  \setbox0=\vbox{
    \global\cdescwidth=0pt
    \def\item[##1]{%
      \settowidth{\dimen0}{##1}%
      \ifdim\dimen0>\cdescwidth
        \global\cdescwidth=\dimen0
      \fi}
    \BODY}
  \global\advance\cdescwidth\labelsep
  \begin{description}[
    style=unboxed,
    align=CenterWithParen,
    labelwidth=\cdescwidth,
    leftmargin=!
  ]
  \BODY
  \end{description}
}

\usepackage{url}

\newtheorem{theorem}{Theorem}[section]

\newtheorem{corollary}[theorem]{Corollary}

\newtheorem{example}[theorem]{Example}

\newtheorem{lemma}[theorem]{Lemma}

\newtheorem{proposition}[theorem]{Proposition}

\newcommand{\DD}{\mathcal{D}}

\newcommand{\PP}{\mathcal{P}}

\newcommand{\darrow}{{\downarrow}}

\newcommand{\restricted}[2]{#1{\mid}_{#2}}

\def\dom{\mathrm{dom}}

\title{Noncommutative Frames Revisited}
\author{Karin Cvetko-Vah}
\address{\textup{(Cvetko-Vah)} Department of Mathematics, Faculty of Mathematics and Physics, University of Ljubljana, Jadranska 21, SI-1000 Ljubljana, Slovenia}
\author{Jens Hemelaer}
\thanks{Jens Hemelaer was supported in part by a PhD fellowship of the Research Foundation (Flanders) and in part by the University of Antwerp (BOF)}
\address{\textup{(Hemelaer)} Department of Mathematics, University of Antwerp, Middelheimlaan 1, B-2020 Antwerp, Belgium}
\author{Jonathan Leech}
\address{\textup{(Leech)} Department of Mathematics, Westmont College, Santa Barbara, CA 93108, USA}

\AtBeginDocument{%
   \def\MR#1{}
}

\begin{document}

\begin{abstract}
In this note, we correct an error in \cite[Theorem 4.4]{kcv-frames} by adding an additional assumption of join completeness. We demonstrate with examples why this assumption is necessary, and discuss how join completeness relates to other properties of a skew lattice. 
\end{abstract}

\maketitle

\section{Introduction}

In \cite{kcv-frames}, the first author introduced noncommutative frames, motivated by a noncommutative topology constructed by Le Bruyn \cite{llb-covers-general-version} on the points of the Connes--Consani Arithmetic Site \cite{connes-consani}, \cite{connes-consani-geometry-as}. The definition of noncommutative frame fits in the general theory of skew lattices, a theory that goes back to Pascual Jordan \cite{jordan} and is an active research topic starting with a series of papers of the third author \cite{jl-rings} \cite{jl-boolean} \cite{jl-normal}. For an overview of the primary results on skew lattices, we refer the reader to \cite{jl-journey} or the earlier systematic survey \cite{jl-survey}.

The main purpose of this note is to study aspects of completeness for certain types of skew lattices and then correct Theorem 4.4 of \cite{kcv-frames}. The corrected version states that if $S$ is a join complete, strongly distributive skew lattice with $0$, then $S$ is a noncommutative frame if and only if its commutative shadow $S/\mathcal{D}$ is a frame. It will appear as Theorem \ref{th:ncframes} here. Originally in \cite{kcv-frames} this was stated for $S$ not necessarily join complete, but we will show in Example \ref{ex:ex-but-not-jc} and \ref{ex:ls-but-not-ex} that this assumption is necessary.

\section{Preliminaries}

A \emph{skew lattice} is a set $A$ endowed with a pair of idempotent, associative operations  $\land$ and $\lor$ which satisfy the absorption laws:
\[
x\land (x\lor y)=x=x\lor (x\land y)\text{ and } (x\land y)\lor y=y=(x\lor y)\land y.
\]  
The terms meet and join are still used for $\land$ and $\lor$, but without assuming commutativity.  Given skew lattices $A$ and $B$, a \emph{homomorphism} of skew lattices is a map $f:A\to B$ that preserves finite meets and joins, i.e.\ it satisfies the following pair of axioms:
\begin{itemize}
\item $f(a\land b)=f(a)\land f(b)$, for all $a,b\in A$;
\item $f(a\lor b)=f(a)\lor f(b)$, for all $a,b\in A$.
\end{itemize}
A \emph{natural partial order} is defined on any skew lattice $A$ by: $a\leq b$ iff $a\land b=b\land a=a$, or equivalently, $a\lor b=b=b\lor a$. The \emph{Green's equivalence relation $\DD$} is defined on $A$ by: $a\DD b$ iff $a\land b\land a=a$ and $b\land a\land b=b$, or equivalently, $a\lor b\lor a=a$ and $b\lor a\lor b=b$. By  Leech's First Decomposition Theorem \cite{jl-rings}, relation $\DD$ is a congruence on a skew lattice $A$ and $A/\DD$ is a maximal lattice image of $A$, also referred to as the \emph{commutative shadow} of $A$. 

Skew lattices are always \emph{regular} in that they satisfy the identities:
\[
a\land x\land a\land y\land a=a\land x\land y\land a \text{ and }
a\lor x\lor a\lor y\lor a=a\lor x\lor y\lor a.
\]
The following result is an easy consequence of regularity.
\begin{lemma}\label{lemma:reg}
Let $a,b,u,v$ be elements of a skew lattice $A$ such that $\DD_u\leq \DD_a$, $\DD_u\leq \DD_b$, $\DD_a\leq \DD_v$ and $\DD_b\leq \DD_v$. Then: 
\begin{enumerate}
\item $a\land v\land b=a\land b$,
\item $a\lor u\lor b=a\lor b$.
\end{enumerate}
\end{lemma}

A skew lattice is \emph{strongly distributive} if it satisfies the identities:
\[
(x\lor y)\land z = (x\land z)\lor (y\land z) \text{ and } x\land (y\lor z)=(x\land y)\lor (x\land z).
\]
By a result of Leech \cite{jl-normal}, a skew lattice is strongly distributive if and only if it is symmetric, distributive and normal, where a skew lattice $A$ is called:
\begin{itemize}
 \item \emph{symmetric} if for any $x,y\in A$, $x\lor y=y\lor x$ iff $x\land y=y\land x$;
 \item \emph{distributive} if it satisfies  the identities:
\begin{eqnarray*}
x\land (y\lor z)\land x=(x\land y\land x)\lor (x\land z\land x)\\
x\lor (y\land z)\lor x=(x\lor y\lor x)\land (x\lor z\lor x);
\end{eqnarray*}
\item \emph{normal} if it satisfies the identity $x\land y\land z\land x=x\land z\land y\land x$.
\end{itemize}
Further, it is shown in \cite{jl-normal} that a skew lattice $A$ is normal if and only if given any $a\in A$ the set
\[
a\!\downarrow~ =\{u\in A\,|\, u\leq a\}
\]
is a lattice. For this reason, normal skew lattices are sometimes called \emph{local lattices}. Given any comparable $\DD$-classes $D<C $ in a normal skew lattice $A$ and any $c\in C$ there exist a unique $d\in D$ such that $d<c$ with respect to the natural partial order.

Finally, a \emph{skew lattice with $0$} is a skew lattice with a distinguished element $0$ satisfying $x\lor 0=x=0\lor x$, or equivalently, $x\land 0=0=0\land x$.

\begin{example}\label{ex:partial-func}
Let $A,B$ be non-empty sets and denote by $\PP(A,B)$ the set of all partial functions from $A$ to $B$. We define the following operations on $\PP(A,B)$:
\begin{eqnarray*}
f \land g & = &\restricted{f}{\dom(f) \cap \dom(g)}\\
f \lor g& = &g\cup \restricted{f}{\dom(f) \setminus \dom(g)}.
\end{eqnarray*}
Leech \cite{jl-normal} proved that $(\PP(A,B);\land, \lor)$ is a strongly distributive skew lattice with $0$. Moreover, given $f,g\in (\PP(A,B);\land, \lor)$ the following hold:
\begin{itemize}
\item $f\, \DD\, g$ iff $\dom(f)=\dom(g)$;
\item $f\leq g$ iff $f=\restricted{g}{\dom(f) \cap \dom(g)}$;
\item $\PP(A,B)/\DD\cong \PP(A)$, the Boolean algebra of subsets of A;
\item $\PP(A,B)$ is left-handed in that $x \land y \land x = x \land y$ and dually, $x \lor y \lor x = y \lor x$ hold.
\end{itemize}
\end{example}

A \emph{commuting subset} of a skew lattice $A$ is a nonempty subset $\{x_i\,|\, i\in I\}\subseteq A$ such that $x_i\land x_j=x_j\land x_i$ and  $x_i\lor x_j=x_j\lor x_i$ hold for all $i,j\in I$. The following result is a direct consequence of the definitions.

\begin{lemma}\label{lemma:hmf-comm}
Let $A$ and $B$ be skew lattices, $f:A\to B$ be a homomorphism of skew lattices, and $\{x_i\,|\, i\in I\}\subseteq A$ be a commuting subset of $A$. Then $\{f(x_i)\,|\, i\in I\} $ is a commuting subset of $B$.
\end{lemma}

A  skew lattice is said to be \emph{join [meet] complete} if all  commuting subsets have suprema [infima] with respect to the natural partial ordering. By a result of Leech \cite{jl-boolean}, the choice axiom implies that any join complete  symmetric skew lattice has a top $\mathcal D$-class. If it occurs, we denote the top $\DD$-class of a skew lattice $A$ by $T$ (or $T_A$). Dually, if $A$ is a meet complete symmetric skew lattice, then it always has a bottom $\DD$-class, denoted by $B$ (or $B_A)$.

A \emph{frame} is a lattice that has all joins (finite and infinite), and satisfies the infinite distributive law:
\[x\land \bigvee_i y_i=\bigvee_i(x\land y_i).
\]
A \emph{noncommutative frame} is a strongly distributive, join complete skew lattice $A$ with $0$ that satisfies the infinite distributive laws:
\begin{equation}\label{eq:inf-dist-law}
 (\bigvee_i x_i)\land y=\bigvee_i (x_i\land y) \qquad \text{and} \qquad
x\land ({\bigvee_i y_i})=\bigvee_i (x\land y_i)
\end{equation}
for all $x,y\in A$ and all commuting subsets $\{x_i\,|\, i\in I\}, \{y_i\,|\, i\in I\}\subseteq A$.  

By a result of Bignall and Leech \cite{jl-discriminator}, any join complete, normal skew lattice $A$ with $0$ (for instance, any noncommutative frame) satisfies the following:
\begin{itemize}
\item $A$ is meet complete, with the meet of a commuting subset $C$ denoted by $\bigwedge C$;
\item any nonempty subset $C\subseteq A$ has an infimum with respect to the natural partial order, to be denoted by $\bigcap C$ (or  by $x\cap y$ in the case $C=\{x,y\}$);
\item if $C$ is a nonempty commuting subset of $A$, then $\bigwedge C=\bigcap C$.
\end{itemize}
We call the $\bigcap C$ the \emph{intersection} of $C$. 

A \emph{lattice section} $L$ of a skew lattice $S$ is a subalgebra that is a lattice (i.e.\ both $\land$ and $\lor$ are commutative on $L$) and that intersects each $\DD$-class in exactly one element.  When it exists, a lattice section is a maximal commuting subset and it is isomorphic to the maximal lattice image, as shown by Leech in \cite{jl-rings}. If a normal skew lattice $S$ has a top $\DD$-class $T$ then given $t\in T$,  $t\darrow =\{x\in S\;|\; x\leq t\}$ is a lattice section of $S$; moreover, all lattice sections are of the form $t\darrow$ for some $t\in T$. Further, it is shown in \cite{jl-rings} that any symmetric skew lattice $S$ such that $S/\DD$ is countable has a lattice section. 
 
 We say that a commuting subset $C$ in a symmetric skew lattice $S$ \emph{extends to a lattice section} if there exists a lattice section $L$ of $C$ such that $C\subseteq L$.

\section{Comparison of completeness properties}

Let $S$ be a normal, symmetric skew lattice. We will consider the following four properties that $S$ might have:
\begin{cdesc}
\item[(JC)] $S$ is join complete;
\item[(BA)] $S$ is \emph{bounded from above}, i.e.\ for every commuting subset $C$ there is an element $s \in S$ such that $c \leq s$ for all $c \in C$;
\item[(EX)] every commuting subset extends to a lattice section;
\item[(LS)] there exists a lattice section.
\end{cdesc}
Note that the last two properties are trivially satisfied if $S$ is commutative.

\begin{proposition}
For normal, symmetric skew lattices, the following implications hold:
\begin{equation*}
\text{\bf (JC)} \Rightarrow \text{\bf (BA)} \Rightarrow \text{\bf (EX)} \Rightarrow \text{\bf (LS)}.
\end{equation*}
\end{proposition}
\begin{proof}
We only prove $\text{\bf (BA)} \Rightarrow \text{\bf (EX)}$, the other two implications are trivial. Take a normal, symmetric skew lattice $S$, such that every commutative subset has a join. Let $C \subseteq S$ be a commuting subset. We have to prove that $C$ extends to a lattice section. For every chain $C_0 \subseteq C_1 \subseteq \dots$ of commuting subsets, the union $\bigcup_{i = 0}^\infty C_i$ is again a commuting subset. So by Zorn's Lemma, $C$ is contained in a maximal commuting subset $C'$. Take an element $s \in S$ such that $s \geq c$ for all $c \in C'$. Then $s\!\downarrow$ contains $C'$ and it is a commuting subset because $S$ is normal. By maximality, $C' = s\!\downarrow$. Again by maximality, $s$ is a maximal element for the natural partial order on $S$. This also means that $s$ is in the top $\mathcal{D}$-class (if $y \in S$ has a $\mathcal{D}$-class with $[y] \not\leq [s]$, then $s \vee y \vee s > s$, a contradiction). So $C'$ is a lattice section.
\end{proof}

We claim that the converse implications do not hold in general. We will give a counterexample to all three of them. In each case, the counterexamples are strongly distributive skew lattices with $0$.

\begin{example}[\textbf{(BA)} $\not\Rightarrow$ \textbf{(JC)}] \label{ex:ex-but-not-jc}
Consider the set $S = \mathbb{N} \cup \{ \infty_a, \infty_b \}$ and turn $S$ into a skew lattice by setting
\begin{gather*}
x \wedge y = \min(x,y) \qquad x \vee y = \max(x,y)
\end{gather*}
whenever $x$ or $y$ is in $\mathbb{N}$ ($\infty_a$ and $\infty_b$ are both greater than every natural number), and
\begin{gather*}
\infty_a \wedge \infty_b = \infty_a = \infty_b \vee \infty_a \\
\infty_b \wedge \infty_a = \infty_b = \infty_a \vee \infty_b.
\end{gather*}
Then $S$ is a left-handed strongly distributive skew lattice with $0$. The commuting subsets of $S$ are precisely the subsets that do not contain both $\infty_a$ and $\infty_b$. Clearly, $S$ is bounded from above (as well as meet complete). However, the commuting subset $\mathbb{N} \subseteq S$ does not have a join.
\end{example}

Note that there are commutative examples as well, for example the real interval $[0,1]$ with join and meet given by respectively maximum and minimum. The element $1$ is an upper bound for every subset, but the lattice is not join complete. However, we preferred an example where the commutative shadow $S/\mathcal{D}$ is join complete. 

\begin{example}[\textbf{(EX)} $\not\Rightarrow$ \textbf{(BA)}] Here we give a commutative example. Take $S = \mathbb{N}$ with the meet and join given by respectively the minimum and maximum of two elements. Then {\textbf{\upshape (EX)}} is satisfied, but {\textbf{\upshape (BA)}} does not hold.
\end{example}

If $S$ satisfies $\textbf{(EX)}$ and $S/\mathcal{D}$ is bounded from above, then for any commuting subset $C \subseteq S$ we can find a lattice section $L \supseteq C$ and an element $y \in L$ such that $[y] \geq [c]$ for all $c \in C$. It follows that $y \geq c$ for all $c \in C$, so $S$ is bounded from above. So any example as the one above essentially reduces to a commutative example.

\begin{example}[\textbf{(LS) $\not\Rightarrow$ \textbf{(EX)}}] \label{ex:ls-but-not-ex}
Consider the subalgebra $S$ of $\mathcal P(\mathbb N, \mathbb N)$ consisting of all partial functions with finite image sets in $\mathbb N$. Note that $S/\mathcal{D} = \PP(\mathbb N)$. The skew lattice $S$ has lattice sections, for example the subalgebra of all functions in $\mathcal P(\mathbb N, \mathbb N)$ whose image set is $\{1\}$.  The set of 1-point functions $\{ n \mapsto n \mid n \in \mathbb N \}$ is clearly a commuting subset, but it cannot be extended to an entire lattice section.
\end{example}

Even the weakest property \textbf{(LS)}, the existence a lattice section, does not always hold for strongly distributive skew lattices.

\begin{example}[\textbf{(LS)} does not hold]
Let $S$ be the subalgebra of $\mathcal P(\mathbb R, \mathbb N)$ consisting of all partial functions $f$ such that $f^{-1}(n)$ is finite for all $n \in \mathbb N$. In particular, if $f \in S$, then the domain of $f$ is at most countable. Conversely, for any at most countable subset $U \subseteq \mathbb R$ we can construct an element $f \in S$ with domain $U$. Suppose now that $Q \subseteq S$ is a lattice section. Then there is an entire function $q : \mathbb{R} \to \mathbb{N}$ such that every $f \in Q$ can be written as a restriction $f = q|_U$ with $U \subseteq \mathbb{R}$ at most countable. Take $n \in \mathbb{N}$ such that $q^{-1}(n)$ is infinite, and take a countably infinite subset $V \subseteq q^{-1}(n)$. Then $q|_V \notin S$, by definition. But this shows that there is no element $f \in Q$ with domain $V$, which contradicts that $Q$ is a lattice section.
\end{example}

By \cite{jl-rings}, any symmetric skew lattice $S$ with $S/\mathcal{D}$ at most countable has a lattice section. This shows that in the above example it is necessary that the commutative shadow $S/\mathcal{D}$ is uncountable.

\section{\texorpdfstring{Join completeness in terms of $\mathcal{D}$-classes}{Join completeness in terms of D-classes}}
\label{sec:joins}

Let $S$ be a normal, symmetric skew lattice. Recall that for an element $a \in S$, we write its $\mathcal{D}$-class as $[a]$. For a $\mathcal{D}$-class $u \leq [a]$, the unique element $b$ with $b \leq a$ and $[b] = u$ will be called the \emph{restriction of $a$ to $u$}. We will denote the restriction of $a$ to $u$ by $a|_u$. For $u,v \leq [a]$ two $\mathcal{D}$-classes, we calculate that
\begin{equation*}
(a|_u)|_v = a|_v \qquad \text{if }v \leq u,
\end{equation*}
and in particular
\begin{equation*}
a|_u \leq a|_v \quad\Leftrightarrow\quad u \leq v.
\end{equation*}

\begin{proposition} \label{prop:joins}
Let $S$ be a normal, symmetric skew lattice and take a commuting subset $\{a_i : i \in I \} \subseteq S$. Then the following are equivalent:
\begin{enumerate}
\item the join $\bigvee_{i \in I} a_i$ exists;
\item the join $\bigvee_{i \in I} [a_i]$ exists and there is a unique $a \in S$ with $[a] = \bigvee_{i \in I} [a_i]$ and $a_i \leq a$ for all $i \in I$.
\end{enumerate}
In this case, $a = \bigvee_{i \in I} a_i$. In particular, $\left[\bigvee_{i \in I} a_i\right] = \bigvee_{i \in I} [a_i]$.
\end{proposition}
\begin{proof}
\underline{$(1)\Rightarrow (2)$}. We claim that $\left[\bigvee_{i \in I} a_i \right]$ is the join of the $\mathcal{D}$-classes $[a_i]$. Because taking $\mathcal{D}$-classes preserves the natural partial order, $[a_i] \leq \left[\bigvee_{i \in I} a_i \right]$ for all $i \in I$. If $\left[\bigvee_{i \in I} a_i \right]$ is not the join of the $[a_i]$'s, then we can find a $\mathcal{D}$-class $u < \left[\bigvee_{i \in I} a_i \right]$ such that $[a_i] \leq u$ for all $i \in I$. But then
\begin{equation*}
a_i \leq \left.\left( \bigvee_{i \in I} a_i \right)\right|_u \!<~ \bigvee_{i \in I} a_i
\end{equation*}
for all $i \in I$, a contradiction. So $\bigvee_{i \in I} [a_i]$ exists and is equal to $\left[ \bigvee_{i \in I} a_i \right]$. For the remaining part of the statement, it is a straightforward calculation to show that $a = \bigvee_{i \in I} a_i$ is the unique element with the given properties.

\underline{$(2)\Rightarrow(1)$}. Write $u = \bigvee_{i \in I} [a_i]$. Let $b \in S$ be an element such that $a_i \leq b$ for all $i \in I$. Then $u \leq [b]$ and $a_i \leq b|_u$ for all $i \in I$. It follows that $a = b|_u$, in particular $a \leq b$. So $a$ is the join of the $a_i$'s.
\end{proof}

\begin{corollary}
Let $S$ be a normal, symmetric skew lattice. Suppose that $S$ is bounded from above and that $S/\mathcal{D}$ is join complete. If every two elements $a,b \in S$ have an infimum $a \cap b$ for the natural partial order, then $S$ is join complete.
\end{corollary}
\begin{proof}
Let $\{ a_i : i \in I\} \subseteq S$ be a commuting subset. Because $S$ is bounded from above, we can take an element $s \in S$ such that $a_i \leq s$ for all $i \in I$. Set $u = \bigvee_{i \in I} [a_i]$. By Proposition \ref{prop:joins} it is enough to show that there is a unique $a \in S$ with $[a] = u$ and $a_i \leq a$ for all $i \in I$. Existence follows by taking the restriction $s|_u$. To show uniqueness, take two elements $a$ and $a'$ with $[a]=[a']=u$ and $a_i \leq a$, $a_i \leq a'$ for all $i \in I$. It follows that $[a \cap a'] = \bigvee_{i \in I} [a_i] = u$. But this shows that $a = a \cap a' = a'$. 
\end{proof}

In Example \ref{ex:ex-but-not-jc}, the two elements $\infty_a$ and $\infty_b$ do not have an infimum.

\section{Noncommutative frames}
The following is a correction of a result in \cite{kcv-frames}, where the assumption of being join complete was erroneously omitted.

\begin{theorem}\label{th:ncframes}
Let $S$ be a join complete, strongly distributive skew lattice with $0$. Then $S$ is a noncommutative frame if and only if $S/\mathcal{D}$ is a frame.
\end{theorem}
\begin{proof}
Suppose that $S/\mathcal{D}$ is a frame. We prove the infinite distributivity laws (\ref{eq:inf-dist-law}). Take $x \in S$ and let $\{y_i : i \in I \} \subseteq S$ be a commuting subset. It is enough to show that
\begin{equation*}
x \wedge \bigvee_{i \in I} y_i ~=~ \bigvee_{i \in I} x \wedge y_i
\end{equation*}
(the proof for the other infinite distributivity law is analogous). Using that $S$ is strongly distributive, it is easy to compute that $y \leq z$ implies $x \wedge y \leq x \wedge z$. In particular, $x \wedge y_i \leq x \wedge \bigvee_{i} y_i$ for all $i \in I$. This shows: 
\begin{equation}\label{eq:inf-dist-inequality}
\bigvee_{i \in I} x \wedge y_i ~\leq~ x \wedge \bigvee_{i \in I} y_i.
\end{equation}
Further, we can use Proposition \ref{prop:joins} to compute
\begin{align*}
\left[ \bigvee_{i \in I} x \wedge y_i \right] ~=~ \bigvee_{i \in I} [x] \wedge [y_i] 
~=~ [x] \wedge \bigvee_{i \in I} [y_i] ~=~ \left[ x \wedge \bigvee_{i \in I} y_i \right],
\end{align*}
where for the middle equality we use that $S/\mathcal{D}$ is a frame. Since left- and right-hand side in (\ref{eq:inf-dist-inequality}) are in the same $\mathcal{D}$-class, the inequality must be an equality, so that $S$ is seen to be a noncommutative frame. Conversely, suppose that $S$ is a noncommutative frame. Then $S$ has a maximal $\mathcal{D}$-class, $T_S$. Let $t$ be in $T_S$. Then $t\!\downarrow$ is a copy of $S/\mathcal{D}$.
\end{proof}

The extra assumption that $S$ is join complete is necessary: the strongly distributive skew lattices from Examples \ref{ex:ex-but-not-jc} and \ref{ex:ls-but-not-ex} have a frame as commutative shadow, but they are not noncommutative frames, since they are not join complete.


\bibliographystyle{amsalphaarxiv}
\bibliography{thesis/thesis}

\end{document}